\documentclass[english, 11pt]{amsart}

\usepackage[english]{babel}
\usepackage[usenames,dvipsnames,svgnames,table]{xcolor}
\usepackage{amsfonts}
\usepackage{enumerate}
\usepackage{amsthm,amsmath,amssymb}
\usepackage{cite}
\usepackage{mathtools}
\usepackage{mdwlist,enumerate}
\usepackage{amscd}
\usepackage{tikz-cd}
\usepackage{thmtools}
\usepackage[margin=3cm]{geometry}
 \usepackage{hyperref}
  \hypersetup{
colorlinks=true,
linkcolor=cyan,
citecolor=cyan
}

\newcommand{\Lring}{\cL_{\text{ring}}}

\def\ac{{\overline{\rm ac}}}
\def\LPres{\cL_{\rm Pres}}

\def\NN{{\mathbb N}}

\def\QQ{{\mathbb Q}}

\def\ZZ{{\mathbb Z}}
\def\acm{{\mathrm{ac}_m}}
\def\ac1{{\mathrm{ac}_1}}
\def\Div{{\mathrm{div}}}
\def\dif{{\mathrm{dif}}}
\def\modus{{\text{mod }}}

\def\cL{{\mathcal L}}

\def\cO{{\mathcal O}}

\def\Br{\mathrm{Br}}
\def\rad{\mathrm{rad}}

\newtheorem{theorem}[subsubsection]{Theorem}
\newtheorem{lemma}[theorem]{Lemma}
\newtheorem{corollary}[theorem]{Corollary}
\newtheorem{proposition}[theorem]{Proposition}

\newtheorem{claim}[theorem]{Claim}

\newtheorem{innercustomthm}{Theorem}
\newenvironment{customthm}[1]
  {\renewcommand\theinnercustomthm{#1}\innercustomthm}
  {\endinnercustomthm}

\theoremstyle{definition}
\newtheorem{definition}[theorem]{Definition}
\newtheorem{convention}[theorem]{Convention}
\newtheorem*{definition*}{Definition}

\theoremstyle{remark}
\newtheorem{remark}[theorem]{Remark}

\usepackage{lineno}

\theoremstyle{plain}

\numberwithin{equation}{theorem}

\begin{document}

\setcounter{tocdepth}{1} 

\title[Definable completeness of $P$-minimal fields and applications]{Definable completeness of $P$-minimal fields and applications}

\author[Pablo Cubides Kovascics]{Pablo Cubides Kovacsics}
\address{Pablo Cubides Kovacsics, Mathematisches Institut der Heinrich-Heine-Universit\"at D\"usseldorf, 
Universit\"atsstr. 1, 40225 D\"usseldorf, Germany. }
\email{cubidesk@hhu.de}

\author[Fran\c coise Delon]{Fran\c coise Delon}
\address{Fran\c coise Delon, Universit\'e de Paris and Sorbonne Universit\'e, CNRS, Institut de Math\'ematiques de Jussieu-Paris Rive Gauche, F-75006 Paris, France.}
\email{delon@math.univ-paris-diderot.fr}

\begin{abstract} We show that every definable nested family of closed and bounded subsets of a $P$-minimal field $K$ has non-empty intersection. As an application we answer a question of Darni\`ere and Halupczok showing that $P$-minimal fields satisfy the ``extreme value property'': for every closed and bounded subset $U\subseteq K$ and every interpretable continuous function $f\colon  U \to \Gamma_K$ (where $\Gamma_K$ denotes the value group), $f(U)$ admits a maximal value. Two further corollaries are obtained as a consequence of their work. The first one shows that every interpretable subset of $K\times\Gamma_K^n$ is already interpretable in the language of rings, answering a question of Cluckers and Halupczok. This implies in particular that every $P$-minimal field is polynomially bounded. The second one characterizes those $P$-minimal fields satisfying a classical cell preparation theorem as those having definable Skolem functions, generalizing a result of Mourgues. 
\end{abstract}

\maketitle

A celebrated result of Miller \cite{miller94} shows that every o-minimal expansion of the real field is either polynomially bounded or the exponential function is definable in it. In contrast, it follows from the work of Darni\`ere and Halupczok \cite{darniereETAL:2015} that every $P$-minimal expansion of $\mathbb{Q}_p$ is polynomially bounded. In fact, they showed more generally that every $P$-minimal expansion of $\mathbb{Q}_p$ is \emph{relatively $P$-minimal}, that is, every interpretable subset of $\mathbb{Q}_p\times\mathbb{Z}^n$ (where $\mathbb{Z}$ stands here for the value group) is already interpretable in the language of rings. However, the question whether every $P$-minimal field is relatively $P$-minimal remained open. We settle this question as a consequence of the following strong form of definable completeness for $P$-minimal fields, which yields in particular that all $P$-minimal fields are polynomially bounded.

\begin{customthm}{(A)}\label{thm:A} Let $K$ be a $P$-minimal field. Every definable nested family of closed and bounded subsets of $K$ has non-empty intersection.  
\end{customthm}

Let us start by putting the previous theorem in context. Recall that a valued field $(K,v)$ is spherically complete if every nested family of balls has non-empty intersection. It is complete if the same condition holds for nested families of balls for which the set of radii is cofinal in the value group of $(K,v)$. Examples of spherically complete fields include all locally compact valued fields and Hahn fields like $K(\!(t^\mathbb{R})\!)$ for any field $K$. The field $\mathbb{C}_p$ is an example of a complete but not spherically complete valued field. 

For first order expansions of a valued field $(K,v)$, \emph{definable completeness} and \emph{definable spherical completeness} correspond to the analogous conditions restricted to \emph{definable} nested families of balls. These are weaker conditions: without being spherically complete, $\mathbb{C}_p$ is definably spherically complete in the language of valued fields $\cL_{\Div}=(+,-,\cdot,0,1,\Div)$, where the binary predicate $\Div(x,y)$ is interpreted by $v(x)\leq v(y)$. Respectively, any countable elementary substructure of $\mathbb{C}_p$ in $\cL_{\Div}$ is definably complete but not complete.

Since definable spherical completeness (and definable completeness) is first-order expressible, it is not difficult to see that all $p$-adically closed and all algebraically closed valued fields are definably spherically complete in $\cL_{\Div}$. It is therefore natural to ask whether these properties are preserved in tame expansions of such fields. Concerning algebraically closed valued fields, the second author showed in \cite{delonCorps:12} that there are $C$-minimal expansions of algebraically closed valued fields which are not even definably complete. As shows Theorem \ref{thm:A}, this does not arise in $P$-minimal expansions of $p$-adically closed fields, which shows a strong difference between these two notions of minimality.

The idea of considering definable nested families of closed and bounded sets (instead of just balls) can be traced back to the work of Miller \cite{miller2001} on definable completeness in ordered structures. To briefly explain how Theorem \ref{thm:A} relates to other properties of $P$-minimal fields and how it is used to settle some open questions in this area, let us first give some informal background on cell decomposition and cell preparation. All formal definitions will be given in Section 1. 


Let $K$ be a $P$-minimal field, $\Gamma_K$ denote the value group of $K$ and $v\colon K \to \Gamma_K\cup\{\infty\}$ denote the valuation map. In \cite{mou-09}, Mourgues 
characterized the class of $P$-minimal fields satisfying a classical cell decomposition theorem as the class of $P$-minimal fields having definable Skolem functions (see later Theorem \ref{thm:mou}). Keeping the discussion informal, by classical cell decomposition we mean that every definable set can be decomposed into finitely many cells which are defined in the spirit of Denef's classical definition in \cite{denef-86}. In his original result, Denef proved more than just a cell decomposition result as he also partitioned the domain of a definable function into finitely many cells in which the function satisfies further properties. Although Denef did not use this terminology, we will make the distinction and call this second and a priori stronger result about definable functions \emph{classical cell preparation}. 

After \cite{mou-09}, it remained open if the class of $P$-minimal fields having definable Skolem functions further satisfies a classical cell preparation theorem. In \cite{darniereETAL:2015}, Darni\`ere and Halupczok characterized the class of $P$-minimal fields satisfying such a preparation theorem as the class of $P$-minimal fields having definable Skolem functions and satisfying the following additional property (see later Theorem \ref{thm:DH}).
 
\begin{definition*}[Extreme value property]\label{def:evp} For every closed and bounded subset $U\subseteq K$ and every interpretable continuous function $f\colon  U \to \Gamma_K$, $f(U)$ admits a maximal value.  
\end{definition*} 

Cell-preparation was obtained in \cite{darniereETAL:2015} by first showing that $P$-minimal fields with the extreme value property are relatively $P$-minimal, i.e., every interpretable subset of $K\times \Gamma_K^n$ is interpretable in $\Lring$. Although the extreme value property can be easily verified for $P$-minimal expansions of $\mathbb{Q}_p$, 
it remained unknown whether the extreme value property and/or relative $P$-minimality hold in every $P$-minimal field (or even in every $P$-minimal field with definable Skolem functions). We use Theorem \ref{thm:A} precisely to show that every $P$-minimal field has the extreme value property. 

\begin{customthm}{(B)}\label{thm:B} Every $P$-minimal field has the extreme value property.   
\end{customthm}

As a consequence of the results in \cite{darniereETAL:2015}, we obtain thus the following. 

\begin{customthm}{(C)}\label{thm:C} Every $P$-minimal field is relatively $P$-minimal.   
\end{customthm}

As mentioned above, Theorem \ref{thm:C} yields that every $P$-minimal field is polynomially bounded (for a formal definition see the introduction of \cite{cubi-delon-IJM}). We would like to point out that it remains open to know whether every $C$-minimal expansion of an algebraically closed non-trivially valued field is polynomially bounded. Some partial results in this direction appear in \cite{cubi-delon-IJM}, where the authors show that every $C$-minimal expansion of an algebraically closed field with value group $
\mathbb{Q}$ (e.g., $\mathbb{C}_p$, $\overline{\mathbb{F}_p}^{\mathrm{alg}}(\!(t^\mathbb{Q})\!)$) is polynomially bounded. 

The following is another corollary of Theorem \ref{thm:B} and the work of Darni\`ere and Halupczok. 

\begin{customthm}{(D)}\label{thm:D} Let $(K,\cL)$ be a $P$-minimal field. Then the following are equivalent
\begin{enumerate}
\item $(K,\cL)$ has definable Skolem functions
\item $(K,\cL)$ has classical cell preparation. 
\end{enumerate}
\end{customthm}

It is worthy to mention that $P$-minimal fields without definable Skolem functions do exist by a result of the first author and Nguyen \cite{cubi-nguyen}. 

\

The article will be structured as follows. In Section \ref{sec:1} we provide all needed background on $P$-minimality including definitions of cells, cell decomposition and cell preparation. We will follow the terminology from \cite{cham-cubi-leen} and make essential use of the clustered cell decomposition theorem proven there. Definable nested families are introduced in Section \ref{sec:nested families}, where we prove Theorem \ref{thm:A} and its consequences.

\section{Preliminaries}\label{sec:1}

Throughout this article we let $K$ denote a $p$-adically closed field, that is, a field elementarily equivalent to a finite extension of $\QQ_p$ in the language of rings $\Lring$. Note that $\Div$ is $\Lring$-definable in such a field. We let $\Gamma_K$ denote the value group, $v\colon K\to\Gamma_K\cup\{\infty\}$ the valuation map, $\cO_K$ the valuation ring and $k_K$ the residue field. For a subset $Y\subseteq \Gamma_K$ and $\gamma\in \Gamma_K$, we define $Y_{>\gamma}\coloneqq \{\gamma'\in Y: \gamma<\gamma'\}$. Concerning balls, $B_{\gamma}(a)$  denotes the ball around $a$ with radius $\gamma$:
\[
B_{\gamma}(a)\coloneqq  \{ x \in K  :  v(x-a) \geqslant \gamma\}. 
\]
The topological closure of a set $X\subseteq K^n$ is denoted by $cl(X)$. Let $\varpi_K$ be a uniformizer in $K$. For a positive integer $m >0$, write $\acm\colon  K^{\times} \to (\cO_K/\varpi_K^m\cO_K)^{\times}$ for the $m^\text{th}$ \emph{angular component map}, the unique group homomorphism such that $\acm(\varpi_K) =1$ and $\acm(u) \equiv u \mod \varpi_K^m$ for any unit $u \in \cO_K$. Existence, uniqueness and definability of such maps was shown in Lemma 1.3 of \cite{clu-lee-2011}. We extend them to $K$ by setting $\acm(0) = 0$. For positive integers $n,m$, let $Q_{n,m}$ be the set
\[ 
Q_{n,m} \coloneqq  \{ x \in K^{\times}  :  v(x) \equiv 0 \ (\modus n) \text{ and } \acm(x) = 1\}.
\]
Note that for $\lambda\in K^\times$ and $x \in \lambda Q_{n,m}$, $\lambda$ encodes the values of $v(x) (\modus n)$ and $\acm(x)$. 

\

For $\cL$ a language extending $\Lring$, the structure $(K,\cL)$ is \emph{$P$-minimal} if for every structure $(K',\cL)$ elementarily equivalent to $(K,\cL)$, every $\cL$-definable subset $X\subseteq K'$ is $\Lring$-definable. Hereafter, $\cL$-definable means definable with parameters in the language $\cL$. For our purposes, it will be sometimes convenient to work in a two sorted language $\cL_2$ where we include the value group as a new sort in the language of Presburger arithmetic $\LPres\coloneqq {(+,-,<,(\equiv_n)_{n\in\NN^\ast})}$ (for details see \cite[Section 2]{cubi-leen-2015}). We write $(K,\cL_2)$ to indicate that we work in the two-sorted language. The following result of Cluckers shows in particular that $\cL_2$-definable subsets of $\Gamma_K$ are $\LPres$-definable.  

%
%
%
%

\begin{theorem}[Cluckers\cite{clu-presb03}\label{thm:raf} Lemma 2 and Theorem 6]\label{thm:semialgpres}
Let $(K, \cL_2)$ be a $P$-minimal field. The value group is stably embedded and its induced structure is that of a pure $\ZZ$-group. In addition, if $Y \subseteq \Gamma_K^m$ is definable, $v^{-1}(Y)$ is $\Lring$-definable.\qed
\end{theorem}

\begin{remark}\label{pubdd} As a consequence of the previous theorem, every $\cL_2$-definable bounded set $Y\subseteq \Gamma_K$ has a maximal element. Equivalently, if $Y$ has no maximal element, it must be cofinal in $\Gamma_K$. This shows in particular that for $P$-minimal fields, the notions of definable completeness and definable spherical completeness are equivalent. 
\end{remark}

\subsection{Cells, cell decomposition and cell preparation}\label{sec:cells}

From now on we work in a $P$-minimal field $(K,\cL_2)$. By definable we mean $\cL_2$-definable. We will use `and' for logical conjunction since the symbol `$\wedge$' will be reserved for something else (see later Section \ref{sec:tree}). Let $S$ denote a definable parameter set (i.e. a definable subset of some product of sorts which will play the role of parameters). A \emph{$\Gamma_K$-cell condition over $S$} is a formula of the form
\begin{equation}\label{eq:Gammacellcond}\tag{E1}
C(s, \gamma) \coloneqq   s\in S \text{ and } \alpha(s) \ \square_1 \ \gamma \ \square_2 \ \beta(s) \text{ and } \gamma\equiv k \ (\modus n) \,, 
\end{equation} 
where $\alpha, \beta$ are definable functions $S \to \Gamma_K$, squares $\square_1,\square_2$ may denote either $<$ or $\emptyset$ (i.e. `no condition'), $\gamma$ is a variable ranging over $\Gamma_K$ and $0\leqslant k< n$ are two integers. If $S=\emptyset$, then $\alpha,\beta$ simply denote elements of $\Gamma_K$.  A \emph{$\Gamma_K$-cell over $S$} is simply the set of elements satisfying a $\Gamma_K$-cell condition over $S$. 

Let $D\subseteq \Gamma_K$ be a $\Gamma_K$-cell defined by a cell condition $C$ as in (\ref{eq:Gammacellcond}) over $S=\emptyset$ (hence fixing $k$ and $n$). A function $g\colon D\to \Gamma_K$ is said to be \emph{linear} if 
\[
g(\gamma)=\frac{a(\gamma-k)}{n} + \delta,  
\]
where $a\in \mathbb{Z}$ and $\delta\in\Gamma_K$. Using Theorem \ref{thm:raf}, the following is a special case of \cite[Theorem 1]{clu-presb03}: 

\begin{theorem}[Cluckers]\label{thm:raf-function} Let $(K,\cL_2)$ be a $P$-minimal field. Let $g\colon Y\subseteq \Gamma_K\to \Gamma_K$ be a definable function. Then there is a finite partition of $Y$ into $\Gamma_K$-cells $Y_1,\ldots, Y_n$ such that $g_{|Y_i}$ is linear. \qed
\end{theorem}

Let us now define $K$-cells. A \emph{$K$-cell condition $C$ over $S$} is a formula of the form
\begin{equation}\label{eq:Kcellcond}\tag{E2}
C(s, c, t) \coloneqq   s\in S \text{ and } \alpha(s) \ \square_1 \ v(t-c) \ \square_2 \ \beta(s) \text{ and } t-c \in \lambda Q_{n,m}, 
\end{equation}
where $t$ and $c$ are variables over $K$, $\alpha, \beta$ are definable functions $S \to \Gamma_K$, squares $\square_1,\square_2$ may denote either $<$ or $\emptyset$, $\lambda \in K$ and $n,m \in \NN\backslash\{0\}$. The variable $c$ is called the \emph{center of $C$}. A $K$-cell condition $C$ is called a \emph{0-cell} condition, resp. a \emph{$1$-cell} condition if $\lambda =0$, resp. $\lambda \neq 0$. Again, if $S=\emptyset$ then $\alpha,\beta$ denote elements of $\Gamma_K$.  

To define $K$-cells we need the following additional notion. Let $C$ be a $K$-cell condition over $S$. Given a function $\sigma\colon  S\to K$, we let $C^\sigma$ denote the set
\[
C^\sigma\coloneqq  \{(s,t)\in S\times K : C(s,\sigma(s),t)\}.  
\]
For $\Sigma\subseteq S\times K$, we let  
$C^\Sigma$ denote the set
\[
C^{\Sigma}\coloneqq  \{(s,t) \in S\times K  :  (\exists c)(c \in \Sigma_s \text{ and } C(s,c,t))\ \}.
\]

A definable set $\Sigma \subseteq S \times K$ is called a \emph{multi-ball over $S$}, if for every $s\in S$ the fibre $\Sigma_s$ is the union of finitely many balls with the same radius. For an integer $\ell>0$, we say a multi-ball $\Sigma$ over $S$ is \emph{of order $\ell$}, if for every $s\in S$ the fibre $\Sigma_s$ is a union of $\ell$ disjoint balls (with the same radius).  

\begin{definition}\label{def:K-cell}  A \emph{classical $K$-cell over $S$} is a set of the form $C^\sigma$ with $C$ a $K$-cell condition over $S$ and $\sigma\colon S\to K$ a definable function. A \emph{clustered $K$-cell over $S$} is a set of the form $C^\Sigma$ where $\Sigma$ is a multi-ball over $S$ of order $\ell$ for some $\ell>0$. A \emph{$K$-cell over $S$} is either a classical or a clustered $K$-cell over $S$.  
\end{definition}

It is worthy to mention that the definition of clustered $K$-cell given in \cite{cham-cubi-leen} contains further properties which we omitted in Definition \ref{def:K-cell} as we will not need them in our arguments (see \cite[Definition 3.4]{cham-cubi-leen} for more details). We will only need two additional properties which we gather in the following remark. 

\begin{remark}\label{rem:classical-cells} Let $X\subseteq S\times K$ be a definable set and let $X_1,\ldots, X_d$ be a cell decomposition of $X$ over $S$. 
\begin{enumerate}
\item We may suppose that every classical $K$-cell $X_i$ over $S$ is defined by a cell condition $C(s,c(s),t)$ as in \eqref{eq:Kcellcond} such that $\square_2=\emptyset$. Indeed, when $\square_2$ is $<$, we can view $X_i$ as a clustered $K$-cell given by $C^\Sigma$ where $\Sigma$ is 
\[
\Sigma\coloneqq \{(s,y)\in S\times K : (\forall t) (C(s,c(s),t)\leftrightarrow C(s,y,t)). 
\]
which is a multi-ball of order 1. 
\item If $X_i$ is a clustered cell $C^\Sigma$ where $\Sigma$ is a multi-ball of order $\ell$ over $S$ and $C$ is a cell condition as in \eqref{eq:Kcellcond}, we may suppose that the function $\beta(s)$ is bounded by the radius of some (any) ball in $\Sigma_s$ (see also the explanation given \cite{cham-cubi-leen} after Definition 1.4). 
\end{enumerate}
\end{remark}


We can now rephrase Mourgues' main result in \cite{mou-09}, which shows in particular that in the absence of definable Skolem functions, classical cells are not enough to describe definable sets. We say that a (one sorted) $P$-minimal field $(K,\cL)$ has \emph{classical cell decomposition}, if for every integer $n\geqslant 1$, every definable set $X\subseteq K^n$ can be decomposed into finitely many classical $K$-cells. Recall that a structure $M$ has \emph{definable Skolem functions} if for every definable set $X\subseteq M^{n+1}$ there is a definable function $g\colon \pi(X)\to M$ such that $(x,g(x))\in X$ for all $x\in \pi(X)$, where $\pi$ denotes the projection of $M^{n+1}$ onto the first $n$ coordinates.

%
%

\begin{theorem}[Mourgues]\label{thm:mou} Let $(K,\cL)$ be a $P$-minimal field. Then the following are equivalent. 
\begin{enumerate}
\item $(K,\cL)$ has definable Skolem functions;
\item $(K,\cL)$ has classical cell decomposition. \qed
\end{enumerate}
\end{theorem}

The main theorem of \cite{cham-cubi-leen} shows that clustered cells are enough to describe definable subsets of $P$-minimal fields without assuming the existence of Skolem functions. 

\begin{theorem}[Clustered cell decomposition] \label{thm:celldecomp} Let $(K,\cL_2)$ be a $P$-minimal field and $X\subseteq S\times T$ be a definable set where $T$ is either $K$ or $\Gamma_K$. Then $X$ can be decomposed into finitely many $T$-cells over $S$.  \qed
\end{theorem}


Let us now define what classical cell preparation is. Let $C^\sigma$ be a classical $K$-cell over $S$ and $f\colon C^\sigma\to K$ be a definable function. Suppose $C$ is a $K$-cell condition over $S$ as given by the formula in (\ref{eq:Kcellcond}). We say that $f$ \emph{is prepared} if there are an integer $k$ and a definable function $\delta\colon  S \to K$ such that for each $(s, t) \in C^\sigma$

\begin{equation*}
v(f(s, t)) = v(\delta(s))+\frac{k v(t - \sigma(s))+v(\lambda^{-k})}{n}. 
\end{equation*}
When $S=\emptyset$, $\delta$ is assumed to be a single element of $K$ and if $\lambda = 0$, we use as a convention that $k = 0$ and $0^0 = 1$.  

\begin{definition}\label{def:cellprepa} The structure $(K,\cL)$ has \emph{classical cell preparation} if given definable functions $f_j\colon X\subseteq K^n\to K$ for $j=1,\ldots,r$, there exists a finite partition of $X$ into classical $K$-cells $C$ over $K^{n-1}$ such that each function $f_j|C$ is prepared and continuous for each $K$-cell $C$. 
\end{definition} 

Any structure $(K,\cL)$ having classical cell preparation also has classical cell decomposition. Classical cell preparation for $p$-adically closed fields $(K,\Lring)$ was proved by Denef in his foundational article \cite{denef-86}. It was later extended by Cluckers for the sub-analytic language $(K,\cL_{an})$ in \cite{clu-2003} (see \cite{clu-2003} or \cite{cubi-nguyen} for a definition). His result is slightly stronger as he shows moreover that prepared functions may be chosen to be not only continuous but even analytic (and analogously for centers). 
  
We can now formally state the result of Darni\`ere and Halupczok from \cite{darniereETAL:2015} quoted in the introduction as follows (see more precisely \cite[Theorems 1.3 and 5.3 ]{darniereETAL:2015}).  

\begin{theorem}[Darni\`ere-Halupczok]\label{thm:DH} Let $(K,\cL)$ be a $P$-minimal field. The following are equivalent:
\begin{enumerate}
\item $K$ has definable Skolem functions and satisfies the extreme value property;  
\item $K$ has classical cell preparation. \qed
\end{enumerate}
\end{theorem}
%
%
%


We will further need the following result, which corresponds to \cite[Lemma 3.2]{cham-cubi-leen_II}.
 
\begin{lemma}\label{lem:finiteimage} Let $K$ be a $P$-minimal field and $f\colon\Gamma_K\to K$ be a definable function. Then $f$ has finite image. \qed
\end{lemma}

%
%

\subsection{The meet-semi lattice tree of closed balls}\label{sec:tree}

Let $T(K)$ denote the set of closed balls of $K$ with radius in $\Gamma_K\cup\{\infty\}$. Ordered by inclusion, $T(K)$ is a meet semi-lattice tree. We let $x\wedge y$ denote the meet of two elements $x,y\in T(K)$. Given $x\in T(K)$, we let $B(x)$ be the set of elements of $K$ in the closed ball associated with $x$. We let $\rad\colon T(K)\to\Gamma_K\cup\{\infty\}$ denote the radius function, namely, the function sending a point $x\in T(K)$ corresponding to the closed ball $B_\gamma(a)$ to $\gamma$. We will often identify points of $K$ with leaves of $T(K)$ (i.e., those $x\in T(K)$ such that $\rad(x)=\infty$). 

For $a\in K$, the \emph{branch of $a$ in $T(K)$}, in symbols $\Br(a)$, is the set of $x\in T(K)$ such that $a\in B(x)$. Every branch of $T(K)$ with cofinal radii (i.e., a linearly ordered subset $H$ of $T(K)$, maximal with respect to inclusion and such that $\{\rad(x) : x\in H\}$ is cofinal in $\Gamma_K$) can be identified with (the branch of) an element $b$ in the completion $\widehat{K}$ of $K$. We thus extend the notation and write $\Br(b)$ for the branch in $T(K)$ of $b\in \widehat{K}$. 

Note that $T(K)$, $\wedge$ and $\rad$ are interpretable (without parameters) in any valued field. Abusing of terminology, we will speak about definable subsets of $T(K)$ instead of saying ``interpretable subsets''.

%

We finish this section with two slightly technical lemmas. 

\begin{lemma}\label{le lemme} Let $I\subseteq \Gamma_K$ be a cofinal subset which is in addition well-ordered. Let $(x_\gamma)_{\gamma\in I}$ be a sequence of elements in $T(K)$ such that for every $c\in K$, there is $\varepsilon_c\in \Gamma_K$ such that the function $f_c\colon I_{>\varepsilon_c}\to \Gamma_K$ given by $\gamma\mapsto \rad(c\wedge x_\gamma)$ is the trace on $I_{>\varepsilon_c}$ of a definable function on $\Gamma_K$. Then, there is a cofinal subset $I'\subseteq I$ such that one of the following holds: 
\begin{enumerate}
\item $(x_\gamma)_{\gamma\in I'}$ is constant;
\item $f_0|I'$ is strictly decreasing;
\item the set 
\[
J\coloneqq\{\gamma\in I  :  (\forall\delta\in I_{>\gamma})(\exists \gamma'\in I_{>\delta})(\exists \gamma''\in I_{>\gamma'})(
x_\gamma \wedge x_{\gamma'} = x_\gamma \wedge x_{\gamma''}
                        				     < x_{\gamma'} \wedge x_{\gamma''})\}
\]
is cofinal in $\Gamma_K$. 
\end{enumerate}
\end{lemma}

\begin{proof}
For $c\in K$, since $f_c$ is the trace of a definable function, and $I_{>\varepsilon_c}$ is cofinal in $\Gamma_K$, by Theorem \ref{thm:raf-function}, there is a cofinal subset $I'$ of $I$ such that $f_c$ restricted to $I'$ is linear and hence either strictly increasing, strictly decreasing or constant. If $f_c|I'$ is strictly increasing, then (3) would hold.  If $f_c|I'$ is strictly decreasing, then for large enough $\gamma$ we have that $f_c(\gamma)=f_0(\gamma)$, and (2) would hold. Therefore, possibly taking a larger $\varepsilon_c$, we may assume that the function $f_c$ is constant on $I_{>\varepsilon_c}$ for every $c\in K$. Assuming (3) does not hold, let $\gamma\in I$ be such that $I_{\geqslant \gamma}\cap J=\emptyset$. Thus, there is $\delta\in I_{>\gamma}$ such that for every $\gamma',\gamma''\in I$ with $\delta<\gamma'<\gamma''$ either 
\[
x_\gamma \wedge x_{\gamma'} \neq x_\gamma \wedge x_{\gamma''} \text{ or } x_\gamma \wedge x_{\gamma''} \geqslant x_{\gamma'} \wedge x_{\gamma''}. 
\]
Pick any $c\in B(x_\gamma)$. Since $f_c$ is constant on $I_{>\varepsilon_c}$, given $\gamma',\gamma''\in I_{>\varepsilon_c}$ we must have that $x_\gamma \wedge x_{\gamma'} = x_\gamma \wedge x_{\gamma''}$. Therefore, if $m\coloneqq\max\{\delta,\varepsilon_c\}<\gamma'<\gamma''$, then 
\[
x_\gamma \wedge x_{\gamma'} = x_\gamma \wedge x_{\gamma''}= x_{\gamma'} \wedge x_{\gamma''}.
\]
Since the residue field is finite, this can only occur if (1) holds for $I'=I_{>m}$. 
\end{proof}

\begin{lemma}\label{lem:downwards} Let $A\subseteq \Gamma_K\times (T(K)\setminus K)$ be a definable set and let $Y$ be its projection to the $\Gamma_K$-coordinate. Assume that
\begin{enumerate}
\item $Y$ is bounded below and cofinal in $\Gamma_K$;
\item there is a positive integer $\ell$ such that $A_\gamma$ has cardinality $\ell$ for each $\gamma\in Y$; 
\item given $\gamma\in Y$, $\rad(x)=\rad(y)$ and $\rad(x\wedge 0)=\rad(y\wedge 0)$ for all $x,y\in A_\gamma$; 
\item the function $g\colon Y\to\Gamma_K$ given by $\gamma\mapsto \rad(x)$ for some (any) $x\in A_\gamma$ is monotone increasing. 
\end{enumerate}
Then, the image of the function $h\colon Y\to \Gamma_K$ given by $\gamma\mapsto \rad(x\wedge 0)$ for some (any) $x\in A_\gamma$, is bounded below. 
\end{lemma} 
\begin{proof}
Suppose for a contradiction that $h(Y)$ is unbounded below. By Theorem \ref{thm:raf-function}, possibly replacing $Y$ by a cofinal subset, we may suppose that $h$ is linear and strictly decreasing. Consider the definable subset of $K$ 
\[
W\coloneqq \bigcup_{\gamma\in Y} \bigcup_{x\in A_\gamma} B(x).  
\]
By assumption, $W$ contains elements of arbitrarily small valuation. Suppose $D_1,\ldots, D_k$ form a cell decomposition of $W$ (over $\emptyset$) with 
\[
D_i\coloneqq \{x\in K : \alpha_i \ \square_{1,i} \ v(x-a_i) \ \square_{2,i} \ \beta_i \text{ and } x-a_i \in \lambda_i Q_{n_i,m_i}\}.   
\]
For $x\in K$ such that $v(x)\in h(Y)$ and $v(x)<\min_i\{v(a_i), \alpha_i, \beta_i\}$, we have that 
\[
x\in W \text{ if and only if for some $i\in\{1,\ldots,k\}$, $\square_{1,i}=\emptyset$ and  $x \in \lambda_i Q_{n_i,m_i}$}.
\] 
For $m\coloneqq\max\{m_i\}$, there is $\gamma_0\in Y$ such that, for all $\gamma\in Y_{>\gamma_0}$ and all $i\in\{1,\ldots,k\}$
\[
h(\gamma)+m<\min_i\{v(a_i), \alpha_i, \beta_i, g(\gamma_0) \}.
 \]
But then, 
$W\cap (B_{h(\gamma)}(0)\setminus B_{h(\gamma)+1}(0))$ is the union of $\ell$ balls of radius strictly bigger than $h(\gamma)+m$ (since $g$ is increasing) which shows that 
\[
W\cap (B_{h(\gamma)}(0)\setminus B_{h(\gamma)+1}(0)) \neq \bigcup_{i} D_i\cap (B_{h(\gamma)}(0)\setminus B_{h(\gamma)+1}(0))
\]
for sufficiently small values of $h(\gamma)$, which contradicts that $W=\bigcup_{i} D_i$. 
\end{proof}


\section{Nested families and definable completeness}\label{sec:nested families}

\subsection{Definable nested families}\label{sec:denef}

Although the most natural acronym for definable nested families was `denef', avoiding temptation, we will use the shorter `$dnf$'. 

\begin{definition}\label{def:dnf} Let $X\subseteq \Gamma_K\times K$ be a definable set and $\pi_1$ denote the projection onto the first coordinate. We say that $X$ is a \emph{definable nested family}, in short $dnf$, if 
\begin{enumerate}
\item for every $\gamma\in \pi_1(X)$, the fibre $X_\gamma$ is non-empty and
\item $X_{\gamma'}\subseteq X_\gamma$ for every $\gamma,\gamma'\in \pi_1(X)$ such that $\gamma<\gamma'$.
\end{enumerate}
A $dnf$ $X$ is said to be a \emph{strict $dnf$} if moreover 
\begin{enumerate}
\item[(2')] $X_{\gamma'}\subsetneq X_\gamma$ for every $\gamma,\gamma'\in \pi_1(X)$ such that $\gamma<\gamma'$.
\end{enumerate}
\end{definition}

\begin{convention}\label{convention}
Let $X$ be a $dnf$. For $\pi_1$ and $\pi_2$ the projections onto the first and second coordinates, we set 
\[
Y\coloneqq\pi_1(X) \hspace{2cm} Z\coloneqq\pi_2(X).
\]
For a subset $Y'\subseteq Y$, we define the subfamily $X_{|Y'}$ as $X_{|Y'}\coloneqq\{(\gamma,x)\in X : \gamma\in Y'\}$. We say that $X$ has \emph{non-empty intersection} if $\bigcap_{\gamma\in Y} X_\gamma\neq\emptyset$.
We let 
\[
\eta\colon Z\to Y\cup\{+\infty\}
\]
be the definable function given by 
\[
\eta(x)\colon 
\begin{cases}
\gamma & \text{ if } x\in X_\gamma \text{ and } (\forall \gamma'\in Y_{>\gamma})(x\notin X_{\gamma'})\\
+\infty & \text{otherwise},\\
\end{cases}
\]
picking the biggest $\gamma\in Y$ such that $x\in X_\gamma$ if existing, and $+\infty$ if $x$ lies in the intersection of all $X_\gamma$. Finally, 
since $Y$ is $\cL_{Pres}$-definable, there is a definable successor function on $Y$ defined by 
\[
\gamma\mapsto \gamma^+\coloneqq \min\{\gamma'\in Y: \gamma'>\gamma\}.  
\]
\end{convention}

In view of condition (1) in Definition \ref{def:dnf}, if $Y$ has a maximal element then $X$ has non-empty intersection. On the other hand, if $Y$ has no maximal element, by Remark \ref{pubdd} $Y$ is cofinal in $\Gamma_K$. 

\begin{lemma}\label{lem:strict-extract} Let $X$ be a $dnf$ with empty intersection. Then there is a cofinal definable subset $Y'\subseteq Y$ such that $X_{|Y'}$ is a strict $dnf$.  
\end{lemma}

\begin{proof} Consider the definable function $\mu\colon Y\to Y$ defined by 
\[
\mu(\gamma)\coloneqq\min\{\gamma'\in Y: \gamma'\geqslant\gamma \text{ and } X_{{\gamma'}^{+}} \subsetneq X_\gamma \}. 
\]
Since $X$ has empty intersection, $\mu$ is well-defined. Note moreover that $\mu$ is monotone increasing. We show that $Y'\coloneqq\mu(Y)$ satisfies the desired property. Since $\mu(\gamma)\geqslant \gamma$ and $\mu$ is monotone, $Y'$ is cofinal. To show that $X_{|Y'}$ is strict, pick $\mu(\gamma),\mu(\delta)\in Y'$ such that $\mu(\gamma)<\mu(\delta)$ for $\gamma,\delta\in Y$. This implies that $\mu(\gamma)<\delta$ (indeed, arguing by the contrapositive, if $\delta\leqslant\mu(\gamma)$ holds, then $\mu(\delta)\leqslant \mu(\mu(\gamma))=\mu(\gamma)$). Therefore, $X_{\mu(\delta)}\subseteq X_{\delta} \subseteq X_{\mu(\gamma)^+}\subsetneq X_{\mu(\gamma)}$, which shows what we wanted.  
\end{proof}

By cell decomposition in $\Gamma_K$ we obtain as a corollary 

\begin{corollary}\label{cor:gamma} Let $X$ be a $dnf$ with empty intersection. Then, there are integers $k,n\geqslant 1$ and $\alpha\in Y$ such that the $\Gamma_K$-cell
\begin{equation}\label{eq:gammacell}\tag{E3}
C\coloneqq \left\{\gamma \in \Gamma_K \left|\begin{array}{l} \alpha < \gamma \\  \gamma \equiv k \ (\emph{\modus} n) \end{array}\right\}\right.,
\end{equation} 
is a subset of $Y$. Moreover, we may assume $n\geqslant 2$ by replacing it by $2n$. \qed 
\end{corollary}



The next step towards Theorem \ref{thm:A} is to prove the special case in which all fibres are balls, that is, to show that $P$-minimal fields are definably complete.

%

\begin{proposition}\label{prop:defcomplete} Every $P$-minimal field is definably complete, that is, every $dnf$ of balls has non-empty intersection.  
\end{proposition}

\begin{proof}
Suppose not and let $X$ be a $dnf$ which is a counterexample. By Lemma \ref{lem:strict-extract} we may assume that $X$ is a strict $dnf$. Let $\delta\colon  Y\to\Gamma_K$ be the definable function sending $\gamma$ to the radius of the ball $X_\gamma$. Replacing $Y$ by $\delta(Y)$, we may assume that $X_\gamma$ is a ball of radius $\gamma$ for all $\gamma\in Y$. By Corollary \ref{cor:gamma}, we may furthermore assume that $Y$ is a $\Gamma_K$-cell defined as in (\ref{eq:gammacell}) for an integer $n\geqslant 2$. Moreover, our assumptions imply that $\eta(Z)\subseteq Y$, so no element in $Z$ has $+\infty$ as its image. Let $\gamma_0$ be the minimal element in $Y$. Consider the definable set 
\[
W\coloneqq\{x\in X_{\gamma_0}: (\forall y\in X_{\eta(x)^+})(\eta(x)=v(x-y) \text{ and } \ac1(x-y)=1)\}. 
\]
Let us first give a geometrical description of the set $W$. For $q$ equal to the cardinality of the residue field $k_K$, each ball $X_\gamma$ is the disjoint union of exactly $q$ subballs of radius $\gamma+1$. For each $\gamma\in Y$, the set $W$ contains exactly one of these subballs. Figure \ref{fig1} shows a picture of $W$. 
\begin{figure}[h]
\caption{The set $W$ corresponds to the union of the grey balls.}
\begin{center}
\includegraphics[scale=0.7]{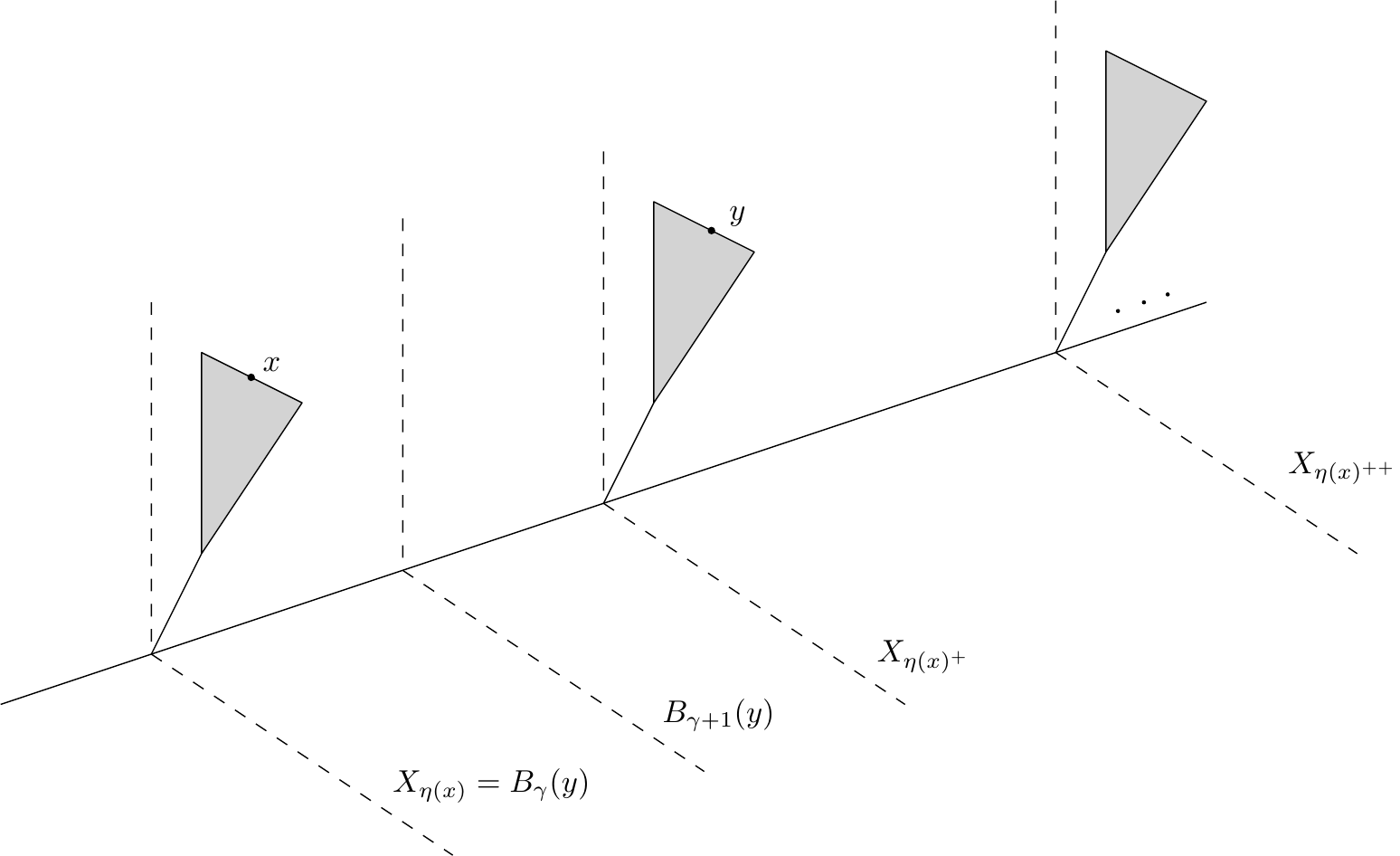}
\end{center}
\label{fig1}
\end{figure}

By $P$-minimality, the set $W$ is $\Lring$-definable, and thus, by Denef's classical cell decomposition, there is a finite set of classical $K$-cells $\mathcal{D}$ such that $W$ is the disjoint union of all $D\in \mathcal{D}$, where 
\[
D \coloneqq  \{x\in K: \alpha_D \ \square_{D,1} \ v(x-\sigma_D) \ \square_{D,2} \ \beta_D\text{ and } x-\sigma_D \in \lambda_D Q_{n_D,m_D}\}, 
\]
with $\alpha_D,\beta_D\in \Gamma_K$, $\sigma_D,\lambda_D\in K$ and $n_D,m_D\in \NN^*$. 

\begin{claim}\label{claim:1} For no $\gamma\in Y$ and no cell $D\in\mathcal{D}$ we have that $X_\gamma\subseteq D$. 
\end{claim}
Suppose $X_\gamma\subseteq D$ for some $\gamma\in Y$ and some cell $D\in\mathcal{D}$. This implies that $X_\gamma\subseteq W$. Let $x\in X_\gamma$ be such that $\eta(x)=\gamma$ and let $y\in X_\gamma$ be such that $y\in X_{\eta(x)^+}$. Then, the ball $B_{\gamma+1}(y)\subseteq W$. By our choice of $Y$ (i.e., $n\geqslant 2$), $\gamma+1\notin Y$, which implies that $W\cap B_{\gamma+1}(y)\neq B_{\gamma+1}(y)$, a contradiction. This shows the claim. 

\

Fix $D\in\mathcal{D}$ and let $\gamma_D\in Y$ be such that $\sigma_D\notin X_{\gamma_D}$ (which exists since otherwise $\sigma_D$ witnesses already that $X$ has non-empty intersection). Since $Y$ is cofinal in $\Gamma_K$ we may further suppose, possibly replacing $\gamma_D$ by a bigger value in $Y$, that for all $x,y\in X_{\gamma_D}$
\begin{equation}\label{eq:1}\tag{E4}
v(\sigma_D-x)+m_D=v(\sigma_D-y)+m_D<v(x-y). 
\end{equation}
Suppose that $X_{\gamma_D}\cap D\neq\emptyset$. In this case, equation (\ref{eq:1}) implies that $X_{\gamma_D}\subseteq D$ contradicting the claim, so $X_{\sigma_D}\cap D=\emptyset$ for every $D\in\mathcal{D}$. To conclude, take $\gamma\in Y$ such that $\gamma>\gamma_D$ for all $D\in\mathcal{D}$, which exists since $Y$ is cofinal in $\Gamma_K$. By construction $X_\gamma$ has empty intersection with every cell $D$, which contradicts that $\mathcal{D}$ is a decomposition of $W$ as $W\cap X_\gamma\neq\emptyset$ for every $\gamma\in Y$.  
\end{proof}

We are ready to show Theorem \ref{thm:A} which we now rephrase: 

\begin{customthm}{(A)}\label{thm:A} Let $X\subseteq \Gamma_K\times K$ be a $dnf$ of closed and bounded sets. Then $\bigcap_{\gamma\in Y} X_\gamma\neq\emptyset$.
\end{customthm}

\begin{proof} By Lemma \ref{lem:strict-extract} we may assume $X$ is a strict $dnf$ and $Y$ is cofinal in $\Gamma_K$ (otherwise the result follows directly). Moreover, we may also suppose that $Z$ is bounded. By clustered cell decomposition (Theorem \ref{thm:celldecomp}), $X$ is equal to a finite disjoint union of $K$-cells $X_1,\ldots, X_d$ over $\Gamma_K$ for some positive integer $d$. Possibly replacing $Y$ by a cofinal subset, we may further assume none of these cells has empty fibres. We obtain the result by a series of cases and reductions. 

\

\emph{Step 1:} We may assume each $X_i$ is a clustered $K$-cell. For suppose $X_i$ is a classical $K$-cell. By Remark \ref{rem:classical-cells} and since $Z$ is bounded, we may assume that 
\[
(\gamma,x)\in X_i \Leftrightarrow \gamma\in Y \text{ and } \ \alpha(\gamma) < \ v(t-c_i(\gamma)) \ \text{ and } t-c_i(\gamma) \in \lambda Q_{n_i,m_i}\,,  
\]
where $c_i\colon Y\to K$ is a definable function. By Lemma \ref{lem:finiteimage}, $c_i$ has finite image, so by taking a cofinal subset of $Y$, we may suppose $c_i$ is constant, say with value $a_0\in K$. Since each $X_\gamma$ is closed, $a_0$ belongs to the intersection of $X$. 

\

\emph{Step 2:} By Step 1, suppose $X_1$ is a clustered $K$-cell with associated multi-ball $\Sigma$ over $\Gamma_K$ of smallest order $\ell\geqslant 1$ among $X_1,\ldots, X_d$, and let $r\geqslant 1$ be the number of clustered cells with associated multi-ball of order $\ell$. Call $(\ell,r)$ the \emph{couple associated to the partition of $X$ into
cells $X_1,\ldots, X_d$}. Note that if $(1,1)$ is the couple associated to some partition of $X$, then $X$ is a $dnf$ of balls and has non-empty intersection by Proposition \ref{prop:defcomplete}. By induction on associated couples (in the lexicographic order), we may further suppose that any other $dnf$ of closed and bounded sets admitting a cell decomposition into clustered $K$-cells with non-empty fibres and smaller associated couple than $(\ell,r)$ has non-empty intersection. Let $A\subseteq Y\times (T(K)\setminus K)$ be the definable set such that for every $\gamma\in Y$, the fibre $A_\gamma$ consists precisely of the set of the $\ell$ closed balls of $\Sigma_\gamma$. Each fibre $A_\gamma$ is thus a finite antichain in $T(K)\setminus K$ such that $\rad(x)=\rad(y)$ for all $x,y\in A_\gamma$ (by definition of multi-ball). In particular, the definable function 
\[
g\colon Y\to \Gamma_K, \hspace{1cm} \gamma\mapsto \rad(x) \text{ for some (any) $x\in A_\gamma$}
\] 
is well-defined. By Theorem \ref{thm:raf-function} and Corollary \ref{cor:gamma}, we may further assume that $g$ is linear and hence either constant or strictly increasing. Note that $g$ cannot be strictly decreasing since $Z$ is bounded. Given $c\in K$, consider the following definable functions
\begin{align*}
& h_c^{\min}\colon Y\to \Gamma_K, &  \hspace{1cm} \gamma\mapsto \min\{\rad(c\wedge x) : x\in A_\gamma\} \\
& h_c^{\max}\colon Y\to \Gamma_K, & \hspace{1cm} \gamma\mapsto \max\{\rad(c\wedge x) : x\in A_\gamma\}  \\
& h_c^{\dif}\colon Y\to \Gamma_K,   & \hspace{1cm}  \gamma\mapsto h_c^{\max}(\gamma)-h_c^{\min}(\gamma).
\end{align*}

\

\emph{Step 3:} We may suppose that for each $c\in K$ there is $\varepsilon_c$ such that $h_c^\dif(\gamma)=0$ for all $\gamma\in Y_{>\varepsilon_c}$. Indeed, if $\ell=1$, then $h_c^{\max}=h_c^{\min}$ and the result is trivial. So suppose $\ell>1$ and that there is $c\in K$ such that the set $Y'=\{\gamma\in Y : h_c^\dif(\gamma)> 0\}$ is cofinal in $Y$. Replacing $Y$ by $Y'$, we may then suppose $h_c^\dif(\gamma)>0$ for all $\gamma\in Y$. But then we can express $\Sigma$ as a disjoint union $\Sigma=\Sigma_0\cup\Sigma_1$ where  
\begin{align*}
&\Sigma_0 \coloneqq \{(\gamma,t) \in Y\times K : t\in B(x), x\in A_\gamma, \rad(x\wedge c) = h_c^{\min}(\gamma) \} \\
&\Sigma_1 \coloneqq \Sigma\setminus \Sigma_0. 
\end{align*}
Both $\Sigma_0$ and $\Sigma_1$ are multi-balls. Possibly replacing $Y$ by a cofinal subset, we may suppose they are multi-balls of fixed orders $\ell_0,\ell_1<\ell$. This shows that we can express $X_1$ as a disjoint union of two clustered $K$-cells with multi-balls of order smaller than $\ell$, and the result follows by induction on associated couples. This shows the claim of this step. Simplifying notation, for each $c\in K$, we let $h_c\colon Y_{>\varepsilon_c} \to \Gamma$ denote the definable function $h_c(\gamma)=h_c^{\min}(\gamma)=h_c^{\max}(\gamma)$. 

\

\emph{Step 4:} Let $I\subseteq Y$ be a cofinal well-ordered subset and $(x_\gamma)_{\gamma\in I}$ be a sequence such that $x_\gamma\in A_\gamma$ for each $\gamma\in I$. By Step 3, the hypotheses of Lemma \ref{le lemme} are satisfied. Indeed, for every $c\in K$, function $f_c\colon I_{>\varepsilon_c}\to \Gamma_K$ given by $\gamma\mapsto \rad(c\wedge x_\gamma)$ is the trace of the definable function $h_c$ above defined. Therefore,  by Lemma \ref{le lemme}, there is a cofinal subset $I'\subseteq I$ such that one of the following holds: 
\begin{enumerate}
\item $(x_\gamma)_{\gamma\in I'}$ is constant;
\item $f_0|I'$ is strictly decreasing;
\item the set 
\[
J\coloneqq\{\gamma\in I  :  (\forall\delta\in I_{>\gamma})(\exists \gamma'\in I_{>\delta})(\exists \gamma''\in I_{>\gamma'})(
x_\gamma \wedge x_{\gamma'} = x_\gamma \wedge x_{\gamma''}
                        				     < x_{\gamma'} \wedge x_{\gamma''})\}
\]
is cofinal in $\Gamma_K$. 
\end{enumerate}

In the remaining steps we deal with each of these cases. 

\

\emph{Step 5:} Suppose (1) holds and let $x$ denote the constant value of $(x_\gamma)_{\gamma\in I'}$. Then, the set $Y'\coloneqq \{\gamma\in Y: x\in A_{\gamma}\}$ is definable and contains $I'$ (so in particular, it is cofinal). Thus, without loss of generality suppose $Y'=Y$. Furthermore, we may suppose $\ell=1$. Indeed, if $\ell>1$, we could express $\Sigma$ as a disjoint union $\Sigma=\Sigma_0\cup\Sigma_1$ where $\Sigma_0=Y\times B(x)$ and $\Sigma_1= \Sigma\setminus \Sigma_0$. Both $\Sigma_0$ and $\Sigma_1$ are multi-balls of smaller order than $\ell$, and the result will follow by induction on associated couples. When $\ell=1$, we have that $\Sigma=Y\times B(x)$ and hence, for all $\gamma\in Y$ 
\[
t\in X_{1,\gamma}\Leftrightarrow (\forall c\in B(x))(\alpha(\gamma) \ < \ v(t-c) \ < \ \beta(\gamma) \text{ and } t-c \in \lambda Q_{n,m}),  
\]
where $\alpha,\beta$ are definable functions and $n,m$ are integers and $\lambda\in K$. By Theorem \ref{thm:raf-function} and possibly replacing $I'$ by a cofinal subset, we may assume that both $\alpha$ and $\beta$ are linear functions. Now, $\beta$ cannot be strictly decreasing since $Z$ is bounded (and no cell has empty fibres). It cannot be strictly increasing either since $\beta(\gamma)<\rad(x)$ (see \ref{rem:classical-cells}). Thus, $\beta$ must be constant. Similarly, $\alpha$ cannot be strictly decreasing since $Z$ is bounded, nor strictly increasing since $\alpha(\gamma)<\beta(\gamma)$ (again, as no cell has empty fibres). Therefore, both $\alpha$ and $\beta$ must be constant functions. But this shows that $X_{1,\gamma}$ is the same set for all $\gamma\in Y$, which yields that any element in $X_{1,\gamma}$ is in the intersection of $X$. 

\

\emph{Step 6:} Let us show (2) cannot hold. For suppose it does. Replacing $Y$ with $Y_{>\varepsilon_0}$ and $I'$ with $I'_{>\varepsilon_0}$, we may suppose $f_0$ is the trace of the definable function $h_0\colon Y\to\Gamma_K$. Since (2) holds, by Theorem \ref{thm:raf-function} and Corollary \ref{cor:gamma}, we may further assume that $h_0$ is strictly decreasing. In particular, $h_0(Y)$ is coinitial in $\Gamma_K$. This contradicts Lemma \ref{lem:downwards}.


\

\emph{Step 7:} Suppose (3) holds. Let us first show that $g$ is strictly increasing. Consider the definable subset of $Y$
\[
Y'\coloneqq \{\gamma\in Y : (\forall \varepsilon \in Y_{>\gamma})(\exists \delta \in Y_{>\varepsilon})(\exists x\in A_\gamma)(\exists y \in A_\varepsilon)(\exists z\in A_\delta)(
x\wedge y<y\wedge z)\}.
\]
By (3), $Y'$ is cofinal in $Y$. Consider the definable subset of $\Gamma_K$ given by 
\[
G\coloneqq \{\rad(x\wedge y)  :  x\in A_\gamma, y\in A_\delta, \gamma,\delta\in Y'\}. 
\]
The set $G$ is definable and, by the choice of $Y'$, it has no maximal element. Then, $G$ is cofinal in $\Gamma_K$, but this cannot be the case if the radius $g$ is constant, as any element in $G$ will be bounded by the constant value of $g$. This shows, $g$ must be strictly increasing. 

Replacing $Y$ by a definable cofinal subset of $Y'$, we may suppose the following: for every $\gamma\in Y$ and every $x\in A_\gamma$, there is $b\in \widehat{K}$ such that for every $\gamma_0\in \Gamma_K$, there are $\varepsilon,\delta\in Y$ with $\gamma<\varepsilon<\delta$, $y\in A_\varepsilon$  and $z\in A_\delta$ such that 
\[
x\wedge y < y\wedge z \text{ and } \gamma_0<\rad(y\wedge z) \text{ and } (x\wedge y)\in \Br(b). 
\]
Indeed, if this condition does not hold for all $x\in A_\gamma$ and all $\gamma$ in a final segment of $Y$, one can again express $\Sigma$ as a disjoint union of two multi-balls of lower order, and the result follows by induction on associated couples. Let $F$ be the set of all such elements $b$ in $\widehat{K}$. We split in two final cases. 

\

\emph{Case 1:} Suppose some $b\in F$ is isolated. Then there is $x_0 \in T(K)\setminus K$ such that  $F\cap B(x_0)=\{b\}$. The set 
\[
\{ x\in T(K)\setminus K : (\exists \gamma \in Y)(x\in A_\gamma\text{ and } x_0<x)\} 
\]
is therefore definable and linearly ordered. Letting $Y'=\{\gamma\in Y :  (\exists x\in A_\gamma)(x_0<x)\}$ and $x_\gamma$ be the unique element in $A_\gamma$ such that $x_0<x_\gamma$, the set  
\[
X'= \bigcup_{\gamma \in Y'} \{\gamma\} \times B(x_\gamma)
\]
is a $dnf$ of balls. By Proposition \ref{prop:defcomplete}, $X'$ has non-empty intersection. But the only element in the intersection must be $b$, so $b\in K$. But then $b$ belongs to the intersection of $X$, since the intersection is a closed set.  

\

\emph{Case 2:} No point $b\in F$ is isolated. Let us show this case does not occur. Note that the cardinality of $F$ is at least the cofinality of $\Gamma_K$. Let $\mu$ be a variable of value group sort and $S_\mu(\Gamma_K)$ denote the set of all types in the variable $\mu$ over $\Gamma_K$. Note that since $\Gamma_K$ is stably embedded (Theorem \ref{thm:raf}), the restriction map $\sigma\colon S_\mu(K\cup \Gamma_K)\to S_\mu(\Gamma_K)$ is a bijection. Let $S_\infty(\Gamma_K)$ be the subset of $S_\mu(\Gamma_K)$ consisting of all completions of the partial type at infinity over $\Gamma_K$ (i.e. the partial type containing the formulas $\{\mu>\gamma: \gamma\in \Gamma_K\}$). An element $p(\mu)\in S_\infty(\Gamma_K)$ is determined by the congruences $\mu\equiv k (\modus n)$ it contains, where $k,n$ are positive integers. This yields that the cardinality of $S_\infty(\Gamma_K)$ is $2^{\aleph_0}$. For each $b\in F$, let $p_b(\mu)$ be an element of $S_\mu(K\cup \Gamma_K)$ containing the set of formulas 
\[
\{(\exists x\in A_\mu)(x>y) : y\in \Br(b) \}\cup \{\mu>\gamma: \gamma\in \Gamma_K\}.  
\]
Let $q_b\in S_\infty(\Gamma_K)$ be the image of $p_b$ under $\sigma$. By possibly working in a large elementary extension, we may suppose that $|\Gamma_K|$ is regular and strictly bigger than $2^{\aleph_0}$. We obtain a contradiction by showing that $|S_\infty(\Gamma_K)|\geq |\Gamma_K|>2^{\aleph_0}$. Assume there is an increasing chain $(F_i)_{i<|\Gamma_K|}$ of subsets of $F$ such that 
\begin{enumerate}
\item $|F_i|<|\Gamma_K|$ for each $i<|\Gamma_K|$;
\item if $b,b'\in F_i$ are different, then $q_b\neq q_{b'}$. 
\end{enumerate}
Setting $F'\coloneqq\bigcup_{i<|\Gamma_K|} F_i$, we have that $|F'|\geqslant |\Gamma_K|$ and $q_b\neq q_{b'}$ for any two elements in $F'$, which shows the above bound. It remains to build the chain. Fix some element $b_0\in F$ and set $F_0=\{b_0\}$. Suppose $F_j$ has been defined for all $j<i$. If $i$ is a limit ordinal, we set $F_i=\bigcup_{j<i}F_j$. So suppose $i=j+1$.  For each $b\in F_j$, let $\gamma_b$ be a realization of $q_b$ (in $\Gamma_L$ for some $K\prec L$). For each $x\in A_{\gamma_b}$ there is at most one element $b'\in F$ such that every $y\in \Br(b')$ lies below $x$. Let $W\subseteq F$ be the set of all such elements $b'\in F$. Since $|W|\leqslant |F_i|\times \ell <|\Gamma_K|\leqslant |F|$, let $b$ be any element in $F\setminus W$ and set $F_{i+1}=F_i\cup \{b\}$. By the choice of $b$, $q_b\neq q_{b'}$ for every $b'\in F_i$. 
\end{proof}

We have now all ingredients to show that every $P$-minimal field satisfies the extreme value property. 

\begin{customthm}{(B)}[Extreme value property] Let $U\subseteq K$ be a closed and bounded set and $f\colon  U \to \Gamma_K$ be a definable continuous function. Then $f(U)$ admits a maximal value.
\end{customthm}

\begin{proof} Let $U\subseteq K$ be closed and bounded and $f\colon  U \to \Gamma_K$ be a definable continuous function. By Remark \ref{pubdd}, if $f(U)$ has no maximal element $\Gamma_K$, then $f(U)$ is cofinal in $\Gamma_K$. For each $\gamma\in f(U)$ let
\[
X_\gamma=cl\left(\bigcup \{ (f^{-1}(\gamma'): \gamma'\in f(U)\text{ and } \gamma\leq\gamma' \}\right), \text{ and }
\]
\[
X=\bigcup_{\gamma\in f(U)} \{\gamma\}\times X_\gamma.
\]
We first show that $X$ is a strict $dnf$ of closed and bounded sets. Each fibre $X_\gamma$ is closed by definition. Since $U$ is closed, $X_\gamma\subseteq U$ for each $\gamma\in f(U)$. Therefore, since $U$ is bounded, so is $X_\gamma$. It remains to show it is nested so let $\gamma,\gamma'\in f(U)$ be such that $\gamma<\gamma'$. By definition of $X$, we trivially have the inclusion $X_{\gamma'}\subseteq X_{\gamma}$. That the inclusion is strict follows by the continuity of $f$. Indeed, let $x\in U$ such that $f(x)=\gamma$, so $x\in X_{\gamma}$. By continuity $f^{-1}(\gamma)$ is open and contains $x$, and has empty intersection with $f^{-1}(\gamma'')$ for all $\gamma''\in f(U)$ such that $\gamma'\leq\gamma''$, hence $x\notin X_{\gamma'}$. This shows that $X$ is a strict $dnf$ of closed and bounded sets. By Theorem \ref{thm:A}, there exists $x\in X_\gamma$ for all $\gamma\in f(U)$. In particular, $x\in U$ so let $f(x)=\gamma_0$ and take $\gamma\in f(U)$ such that $\gamma>\gamma_0$. Since $x\in X_{\gamma}$, there is $\gamma'\geq\gamma$ such that  $x\in cl(f^{-1}(\gamma'))$ which contradicts that $f(x)=\gamma_0$. 
\end{proof}

The following theorem corresponds to \cite[Theorem 4.1]{darniereETAL:2015}.

\begin{theorem}\label{thm:EVPrelPmin} Assume that $(K,\cL)$ is $P$-minimal and satisfies the extreme value property. Then every definable set $X \subseteq \Gamma_K^d \times K$ is $\Lring$-definable, for every $d\geq 0$. \qed
\end{theorem}

Theorems \ref{thm:C} and \ref{thm:D} are direct corollaries of Theorem \ref{thm:B} and Theorems \ref{thm:EVPrelPmin} and \ref{thm:DH}, the latter two due to Darni\`ere and Halupczok in \cite{darniereETAL:2015}. 

\

We finish with a short question. In view of the clustered cell decomposition theorem for general $P$-minimal fields, can one provide an analogue of cell preparation for general $P$-minimal fields?

\subsection*{Acknowledgements:} P. Cubides Kovacsics was partially supported by the ERC project TOSSIBERG (Grant Agreement 637027) and individual research grant
\emph{Archimedische und nicht-archimedische Stratifizierungen h\"oherer
Ordnung}, funded by the DFG. F. Delon was partially supported by the Idex Universit\'e de Paris.

\bibliographystyle{amsplain}
\bibliography{biblio}

\end{document}